\documentclass[oneside]{amsart}
\usepackage{amsmath,amsthm,latexsym,amssymb,hyperref,amsfonts}
\usepackage{stmaryrd}
\usepackage{eucal}
\usepackage{mathrsfs}
\usepackage{comment}
\usepackage[all]{xy}
\usepackage[usenames,dvipsnames]{xcolor}
\usepackage{pdflscape}
\usepackage{tikz}
\allowdisplaybreaks
\usetikzlibrary{matrix,arrows}
\usepackage[textwidth=13.5cm, textheight=21cm]{geometry}

\numberwithin{equation}{section}

\newtheorem{thm}{Theorem}[section]
\newtheorem{lem}[thm]{Lemma}
\newtheorem{prop}[thm]{Proposition}
\newtheorem{cor}[thm]{Corollary}

\theoremstyle{definition}
\newtheorem{defn}[thm]{Definition}

\theoremstyle{remark}

\DeclareMathAlphabet{\mathbbe}{U}{bbold}{m}{n}

\newdir{ >}{{}*!/-10pt/@{>}}

\newcommand\lan{\operatorname{Lan}}
\newcommand\ran{\operatorname{Ran}}

\newcommand{\lra}{\longrightarrow}

\newcommand{\Ch}{\mathsf{Ch}}

\newcommand{\A}{\mathsf A}

\newcommand{\cA}{\mathcal A}

\newcommand{\cF}{\mathcal F}
\newcommand{\cO}{\mathcal O}
\newcommand{\cD}{\mathcal D}
\newcommand{\cC}{\mathcal C}

\newcommand{\bQ}{\mathbb Q}

\newcommand{\bR}{\mathbb R}

\newcommand{\cP}{\mathcal{P}}

\newcommand\adjunct[4]{\xymatrix@C=3pc@R=4pc{#1\ar @<1.25ex>[rr]^{#3}&\perp&#2\ar @<1.25ex>[ll]^{#4}}}

\begin{document}
\title{An algebraic model for rational  $G$--spectra over an exceptional subgroup}

\author[K\c{e}dziorek]{Magdalena K\c{e}dziorek}

\address{MATHGEOM \\
    \'Ecole Polytechnique F\'ed\'erale de Lausanne \\
    CH-1015 Lausanne \\
    Switzerland}
\email{magdalena.kedziorek@epfl.ch}


 \begin{abstract}
We give a simple algebraic model for rational $G$--spectra over an \emph{exceptional} subgroup, for any compact Lie group $G$. Moreover, all our Quillen equivalences are symmetric monoidal, so as a corollary we obtain a monoidal algebraic model for rational $G$--spectra when $G$ is finite. We also present a study of the relationship between induction -- restriction -- coinduction adjunctions  and left Bousfield localizations at idempotents of the rational Burnside ring. 
 \end{abstract}
\maketitle
\tableofcontents
\section{Introduction}

\paragraph{\bf{Modelling the category of rational $G$ -- spectra}}
$G$--spectra are representing objects for cohomology theories designed to take symmetries of spaces into account. Rationalising this category removes topological complexity, but leaves interesting equivariant behaviour. In order to understand this behaviour, we try to find a purely algebraic description of the category, i.e. an algebraic category Quillen equivalent to the category of rational $G$--spectra. As a result the homotopy category of an algebraic model is equivalent to the rational stable $G$ homotopy category, thus all the homotopy information in both is the same. 

Let $G$ be a compact Lie group. We start with the category of rational $G$--spectra, by which we mean the category of $G$--spectra, but with the model structure that is a left Bousfield localization of the stable model structure at the rational sphere spectrum. Thus the weak equivalences are maps which become isomorphisms after applying the rational homotopy group functors, i.e. $\pi_{\ast}^H(-)\otimes \mathbb{Q}$ for all closed subgroups $H$ in $G$.

Therefore we want to find an algebraic category in which calculations should be easier, equipped with a model structure Quillen equivalent to the category of rational $G$--spectra.
Note that the level of accuracy we would like to get is ``up to homotopy", therefore we don't need equivalences of categories (which is usually difficult to get), but (possibly a chain of) Quillen equivalences. If we find such a chain of Quillen equivalences between the category of rational $G$--spectra and some algebraic category we say that we found an ``algebraic model" for rational $G$--spectra. We can then work and perform constructions and calculations in this new setting to get results for the homotopy category of rational $G$--spectra.\\

\paragraph{\textbf{Main result}}
We call a subgroup $H \leq G$ \emph{exceptional} if $N_GH/H$ is finite and $H$ can be completely 
separated from other subgroups of $G$ in a sense that there is an idempotent $e_{(H)_G}$ in the rational Burnside ring $\A(G)_\bQ$ corresponding to the conjugacy class of $H$ in $G$ and $H$ does not contain a subgroup $K$ such that $H/K$ is a (non-trivial) torus (see Definition \ref{defn:exceptional}). 

A category of rational $G$--spectra over an exceptional subgroup $H$ is modelled by the left Bousfield localization at the idempotent $e_{(H)_G}$. This models the homotopy category of rational $G$--spectra with geometric isotropy $H$ and it is a particularly nicely behaved part of rational  $G$--spectra, which in its structure resembles (or generalizes) rational $\Gamma$--spectra for finite $\Gamma$. 

\begin{thm}Suppose $G$ is any compact Lie group and $H$ an exceptional subgroup of $G$. Then there is a zig-zag of symmetric monoidal Quillen equivalences from rational  $G$--spectra over $H$ to $$\Ch(\bQ[N_GH/H]-\mathrm{mod})$$ with the projective model structure.
\end{thm}

If $G$ is finite then every subgroup of $G$ is exceptional and there is finitely many conjugacy classes of subgroups of $G$, so by the splitting result of \cite{Barnes_Splitting} the category of rational $G$--spectra splits into a finite product (over conjugacy classes $(H)$ of subgroups of $G$) of rational $G$--spectra over $H$. Thus our approach gives a monoidal algebraic model for rational $G$--spectra for finite $G$ (see also Section \ref{section:finiteG}).

\begin{cor}Suppose $G$ is a finite group. Then there is a zig-zag of symmetric monoidal Quillen equivalences from rational  $G$--spectra to $${\prod_{(H),H\leq G}\Ch(\bQ[N_GH/H]-\mathrm{mod})}$$ with the (objectwise) projective model structure.
\end{cor}

Results above were obtained using an analysis of the interplay between left Bousfield localizations at idempotents of the rational Burnside ring and the induction -- restriction -- coinduction adjunctions in Section \ref{section:localization}. This analysis is similar in flavour to the one presented in \cite{GreenShipleyFixedPoints} for the inflation--fixed point adjunction. The point is to recognise when these adjunctions become Quillen equivalences in situations that are of intrest to us. Similar analysis is used to obtain an algebraic model for rational $SO(3)$--spectra in \cite{KedziorekSO(3)} and also for the toral part of rational $G$--spectra, for any compact Lie group $G$ in \cite{BGK}. \\

\paragraph{\textbf{Existing work}}
 It is expected that for any compact Lie group $G$ there exists an algebraic category $\cA(G)$ which is Quillen equivalent to that of rational $G$--spectra. 

There are many partial results and examples for specific Lie groups $G$ for which an algebraic model has been given. An algebraic model for rational $G$ equivariant spectra for finite $G$ is described in \cite[Example 5.1.2]{SchwedeShipleyMorita}. It was shown in \cite[Theorem 1.1]{ShipleyHZ} that rational spectra are monoidally Quillen equivalent to chain complexes of $\bQ$ modules.  An algebraic model for rational torus equivariant spectra was presented in \cite{GreenShipleyT-equiv}, whereas a new approach in \cite{BGKS} gives a symmetric monoidal algebraic model for $SO(2)$. Also, an algebraic model for free rational  $G$--spectra was given in \cite{GreenShipFree} for any compact Lie group $G$.

However, there is no algebraic model known for the whole category of rational $G$--spectra for an arbitrary compact Lie group $G$. The present paper establishes the first part of a general result, providing a model for rational $G$--spectra over an exceptional subgroup (see Definition \ref{defn:exceptional}), for any compact Lie group $G$.

The approach to the algebraic model for rational $G$--spectra, where $G$ is finite, in \cite{BarnesFiniteG} requires the use of localizations of commutative ring $G$--spectra. Interpreting it correctly would rely on work of Blumberg and Hill \cite{BlumbergHillL1}. The method of this paper avoids such subtleties and thus presents a more immediate and easier proof of a zig-zag of symmetric monoidal Quillen equivalences in the case where $G$ is a finite group.\\

\paragraph{\textbf{Outline of the paper}}
This paper is structured as follows. In Section \ref{section:subgroups} we describe subgroups of a compact Lie group $G$ and discuss some related idempotents in the rational Burnside ring of $G$. Section \ref{section:spectra} recalls basic properties of $G$ orthogonal spectra that we will use later on. In Section \ref{section:localization} we link the different behaviour of subgroups of $G$ with the left Bousfield localization and the induction -- restriction -- conduction adjunctions. This is the heart of the paper and it allows us to  provide a zig-zag of symmetric monoidal Quillen equivalences in Section \ref{section:monoidal} that instead of using the Morita equivalences presented in \cite{SchwedeShipleyMorita} uses the inflation--fixed point adjunction which is strong symmetric monoidal. \\

\paragraph{\textbf{Notation}}We will stick to the convention of drawing the left adjoint above the right one in any adjoint pair.\\

\paragraph{\textbf{Acknowledgments}}
This is a part of my PhD thesis (under the supervision of John Greenlees) and I would like to thank John Greenlees and Dave Barnes for many useful discussions and comments.

\section{Subgroups of a Lie group G}\label{section:subgroups}

Recall that for $H \leq G$,  $N_GH=\{g \in G\ |\ gH=Hg\}$ is the normaliser of $H$ in $G$. We use the notation $W=W_GH=N_GH/H$ for the Weyl group of $H$ in $G$.

Suppose $\cF(G)$  is a space of closed subgroups of $G$ with finite index in their normalizer (i.e. $H\leq G$ such that $N_GH/H$ is finite) considered with topology given by the Hausdorff metric. 
By the result of tom Dieck (\cite[5.6.4, 5.9.13]{tomDieck} $\A(G)\otimes \bQ =C(\cF(G)/G,\bQ)$, where $C(\cF(G)/G,\bQ)$ denotes the ring of continuous functions on the orbit space $\cF(G)/G$ with values in discrete space $\bQ$. From now on we will use notation $\A(G)_\bQ$ for  $\A(G)\otimes \bQ$. It is clear that idempotents of the rational Burnside ring of $G$ correspond to the characteristic functions on open and closed subspaces of the orbit space  $\cF(G)/G$ (or equivalently, to open and closed subspaces of $\cF(G)/G$). Thus it makes sense to refer to an idempotent $e_{V}$, i.e. the one corresponding to the subspace $V$ in $\cF(G)/G$, provided that $V$ is open and closed in $\cF(G)/G$.

Every inclusion $i: H \lra G$ gives a ring homomorphism $i^*:\A(H)_\bQ \lra \A(G)_\bQ$. To see what is an image of an idempotent under $i^*$ it is better to relate idempotents to the subspaces of the space of all closed subgroups of $G$ as follows. Suppose $Sub_f(G)$ is the topological space of all closed subgroups of $G$ with the $f$-topology (see \cite[Section 8]{GreeRationalMackey1} for details). One can relate an idempotent in a rational Burnside ring $\A(G)_\bQ$ to an open and closed, $G$--invariant subspace of $ Sub_f(G)$ which is a union of $\sim$--equivalence classes (here $\sim$ denotes the equivalence relation generated by $H \sim G$, where $H \sim G$ iff $H \leq G$ and $G/H$ is a torus). 

With this in mind, we give the following

\begin{defn}\label{defn:exceptional}
Suppose $G$ is a compact Lie group. We say that a closed subgroup $H \leq G$ is \emph{exceptional} in $G$ if $W_GH$ is finite, there exist an idempotent $e_{(H)_G}$ in the rational Burnside ring of $G$ corresponding to the conjugacy class of $H$ in $G$ and $H$ does not contain any subgroup cotoral in $H$, where a subgroup $K \leq H$ is cotoral in $H$ if $H/K$ is a (non-trivial) torus.
\end{defn}

Note that any subgroup of a finite group $G$ is exceptional. In $O(2)$ only finite dihedral subgroups are exceptional; in particular none of the finite cyclic subgroups is exceptional. In $SO(3)$ all finite dihedral subgroups are exceptional (except for $D_2$), but we have more: there are four more conjugacy classes of exceptional subgroups: $A_4,\Sigma_4, A_5$ and $SO(3)$, where $A_4$ denotes rotations of a tetrahedron, $\Sigma_4$ denotes rotations of a cube and $A_5$ denotes rotations of a dodecahedron.

We introduced the notion of an exceptional subgroup $H$ because we will use the corresponding idempotent in the rational Burnside ring to split the category of rational  $G$--spectra into the part over an exceptional subgroup $H$ and its complement. In this paper we present the model for rational $G$--spectra over an exceptional subgroup $H$. 

On the way towards the algebraic model different subgroups of $G$ will behave slightly differently. This behaviour is closely related to the following

\begin{defn}
\label{goodSub}
Suppose $H,K$ are closed subgroups of $G$ such that $K$ is exceptional in $G$. Suppose further that $i:H \lra G$ is an inclusion. We say that $K$ is $H$--{\emph{good}} in $G$ if $i^*(e_{(K)_G})= e_{(K)_H}$ and  $H$--{\emph {bad}} in $G$ if it is not $H$--good, i.e. $i^*(e_{(K)_G})\neq e_{(K)_H}$.
\end{defn}

There is a definition of good and bad subgroups in \cite[Definition 6.3]{GreenleesSO3}, however it was designed to capture different properties than our definition and thus they are not the same. As an example, if a trivial subgroup is exceptional in $G$ it is always $H$--good in $G$ for any $H \leq G$ according to our definition and $H$--bad according to Greenlees' definition (unless $H$ is normal in $G$).

It is easy to see that any exceptional subgroup $H$ in a compact Lie group $G$ is $H$--good in $G$.

We present some examples.
\begin{lem}\label{goodAndBadSub}For exceptional subgroups in $G=SO(3)$ we have the following relation between $H$ and its normaliser $N_GH$:
\begin{enumerate}
\item $A_5$ is $A_5$--good in $SO(3)$.
\item $\Sigma_4$ is $\Sigma_4$--good in $SO(3)$.
\item $A_4$ is $\Sigma_4$--good in $SO(3)$.
\item $D_4$ is $\Sigma_4$--bad in $SO(3)$.
\end{enumerate}
\end{lem}
\begin{proof} We only need to prove Part 3 and 4. Part 3 follows from the fact that there is one conjugacy class of $A_4$ in $\Sigma_4$, as there is just one subgroup of index 2 in $\Sigma_4$. Part 4 follows from the observation that there are two subgroups of order 4 in $D_8$ (so also in $\Sigma_4$) and they are conjugate by an element $g\in D_{16}$, which is the generating rotation by 45 degrees (thus $g \notin D_8$ and thus $g\notin \Sigma_4$).
\end{proof}

\section{$G$ orthogonal spectra, left Bousfield localization and splitting}\label{section:spectra}
There are many constructions of categories of spectra ($G$--spectra) equipped with model structures, such that the homotopy category is equivalent to the usual stable homotopy category of spectra ($G$--spectra, respectively). However, since we are interested in modelling the smash product as well, we choose to work with a model with a strictly associative monoidal product compatible with model structure so that its homotopy category is equivalent to the usual stable homotopy category with the smash product known in algebraic topology.

When we work with non equivariant spectra, there are several categories having this property, and we choose to work with the category of symmetric spectra defined in \cite{HoveySSSymSpectra} and discussed in details in \cite{SchwedeUntitledProject}. We  will use it briefly towards the end of Section \ref{monoidalH}. Whenever we are interested in modelling the category of $G$--equivariant cohomology theories we choose to work with the category of $G$--orthogonal spectra defined and described in \cite{MandellMay}, for which we use the notation $G-\mathrm{Sp}^\cO$. 

The construction of both categories is similar and we refer the reader to the papers above for details. The idea is to first construct a diagram of spaces (or simplicial sets) indexed by some fixed category. Then to define a tensor product on the category of diagrams and choose a monoid $S$ (sphere spectrum). Spectra are defined to be $S$-modules. Depending on the indexing category we get symmetric spectra or $G$ orthogonal spectra.

In this section we recall briefly some properties of the category of orthogonal $G$--spectra after Chapter II of \cite{MandellMay}.  We stress that unless otherwise stated we use the term of $G$ orthogonal spectra to implicitly mean the ones indexed on a complete $G$ universe. By \cite[Theorem 4.2.]{MandellMay} there is a model structure on orthogonal $G$--spectra, called the stable model structure where a map of orthogonal spectra $f: X \lra Y$ is a weak equivalence if it is a $\pi_{\ast}$--isomorphism (i.e. it is a $\pi_\ast^H$--isomorphism for all $H\leqslant G$). This model structure is cofibrantly generated, stable, monoidal, proper and cellular (see \cite[Theorem III 4.2]{MandellMay}). 

What is more, we have a good way of checking that a map in $G-\mathrm{Sp}^\cO$ is a weak equivalence. For any closed subgroup $H$ in $G$, any orthogonal spectrum $X$ and integers $p \geq 0$ and $q>0$
\begin{equation}\label{generatorsG/H}[\Sigma^pS^0\wedge G/H_+,X]^G \cong \pi^H_p(X) \ \ \ [F_qS^0\wedge G/H_+,X]^G \cong \pi^H_{-q}(X)
\end{equation}
where the left hand sides denote morphisms in the homotopy category of $G-\mathrm{Sp}^\cO$ and $F_q -$ is the left adjoint to the evaluation functor at $\bR^q$, $Ev_{\bR^q}(X)=X(\bR^q)$. 

There is one more property which makes the stable model structure on $G-\mathrm{Sp}^\cO$ easy to work with, namely it has a set of homotopically compact generators.
By \cite[Definition 7.1.1]{Hovey} a homotopy category of a stable model category is triangulated. In this setting we can make the following definitions after \cite[Definition 2.1.2]{SchwedeShipleyMorita}.

\begin{defn}\label{compactobj}
Let $\cC$ be a triangulated category with infinite coproducts. A full triangulated subcategory of $\cC$ (with shift and triangles induced from $\cC$) is called \emph{localizing} if it is closed under coproducts in $\cC$. A set $\cP$ of objects of $\cC$ is called a  \emph{set of generators} if the only localizing subcategory of $\cC$ containing objects of $\cP$ is the whole of $\cC$. An object $X$ in $\cC$ is  \emph{homotopically compact}\footnote{We chose to emphasize the word "homotopically", since there are several different meanings of compactness in the literature.} if for any family of objects $\{A_i\}_{i \in I}$ the canonical map  
$$\bigoplus_{i\in I}[X,A_i]^{\cC} \lra [X, \coprod_{i \in I}A_i]^{\cC}$$ is an isomorphism.
An object of a stable model category is called a homotopically compact generator if it is so when considered as an object of the homotopy category.
\end{defn}

The set of suspensions and desuspensions of $G/H_+$, where $H$ varies through all closed subgroups of $G$, is a set of homotopically compact generators in the stable model category $G-\mathrm{Sp}^\cO$. Those objects are homotopically compact since homotopy groups commute with coproducts and it is clear from \cite[Lemma 2.2.1]{SchwedeShipleyMorita} and  (\ref{generatorsG/H}) that this is a set of generators for $G-\mathrm{Sp}^\cO$.

There is an easy-to-check condition for a Quillen adjunction between stable model categories with sets of homotopically compact generators to be a Quillen equivalence:

\begin{lem}\label{qequiv_generators}Suppose $F: \cC \rightleftarrows \cD : U$ is a Quillen pair between stable model categories with sets of homotopically compact generators, such that the right derived functor $RU$ preserves coproducts (or equivalently, such that the left derived functor sends homotopically small generators to homotopically small objects). Then to know $F,U$ is a Quillen equivalence it is enough to check that a derived unit and counit are weak equivalences for generators. 
\end{lem}
\begin{proof}This follows from the fact that the homotopy category of a stable model category is a triangulated category. As the derived unit and counit conditions are satisfied for a set of objects $\mathcal{K}$ then they are also satisfied for every object in the localizing subcategory for $\mathcal{K}$. Since $\mathcal{K}$ consisted of generators the localizing subcategory for $\mathcal{K}$ is the whole category. 
\end{proof}

Our basic category to work with is the category $G-\mathrm{Sp}^\cO$ of $G$--orthogonal spectra. However, in this paper we are interested only in the homotopy category of rational  $G$--spectra over an exceptional subgroup. Localization is our main tool to make the model category of $G$--spectra easier, so that it models exactly the part that we want. We obtain it by firstly rationalising the stable model category of  $G$--spectra using the localization at an object $S_\bQ$, which is a rational sphere spectrum. Then, we localize it further to extract the behaviour of an exceptional subgroup.

For details on left Bousfield localization at an object we refer the reader to \cite{MandellMay}. We recall the following result, which is \cite[Chapter IV, Theorem 6.3]{MandellMay}

\begin{thm}Suppose $E$ is a cofibrant object in $G-\mathrm{Sp}^\cO$ or a cofibrant based $G$--space. Then there exists a new model structure on the category $G-\mathrm{Sp}^\cO$, where a map $f:X \lra Y$ is
\begin{itemize}
\item[-] a weak equivalence if it is an $E$--equivalence, i.e. $Id_E\wedge f: E\wedge X \lra E\wedge Y$ is a weak equivalence
\item[-] cofibration if it is a cofibration with respect to the stable model structure
\item[-] fibration if it has the right lifting property with respect to all trivial cofibrations.
\end{itemize}
The $E$--fibrant objects $Z$ are the $E$--local objects, i.e. $[f, Z]^G: [Y,Z]^G \lra [X,Z]^G$ is an isomorphism for all $E$--equivalences $f$.  
$E$--fibrant approximation gives Bousfield localization $\lambda : X\lra L_EX$ of $X$ at $E$. 
\end{thm}

We use the notation $L_E(G-\mathrm{Sp}^\cO)$ for the model category described above and will refer to it as a left Bousfield localization of the category of  $G$--spectra at $E$. Notice that if $E$ and $F$ are cofibrant objects in $G-\mathrm{Sp}^\cO$  then the localization first at $E$ and then at $F$ is the same (model category) as the localization at $E\wedge F$.

Recall that, an $E$--equivalence between $E$--local objects is a weak equivalence (see \cite[Theorems 3.2.13 and 3.2.14]{Hirschhorn}).
All our localizations are smashing (see \cite{Ravenel_Localization} for definition) thus they preserve homotopically compact generators (since the fibrant replacement preserves infinite coproducts).

As mentioned above, the first simplification of a category of  $G$--spectra is rationalization, i.e. localization at an object $S_\bQ$, which is a rational sphere spectrum (the Eilenberg--Moore spectrum for $\bQ$, see for example \cite[Definition 5.1]{Barnes_Splitting}). This spectrum has the property that $\pi_*(X \wedge S_\bQ)=\pi_*(W)\otimes \bQ$. 
We refer to this model category as  rational  $G$--spectra. 

The next step on the way towards the algebraic model is to split the category of rational  $G$--spectra using idempotents of the rational Burnside ring $\A(G)_\bQ$. We know that idempotents of the (rational) Burnside ring split the homotopy category of (rational) $G$--spectra. Barnes' result \cite{Barnes_Splitting} allows us to perform a compatible splitting at the level of model categories. We want to use the idempotent $e_{(H)_G}$ corresponding to the exceptional subgroup $H$ in $G$ (see Definition \ref{defn:exceptional}) and the idempotent corresponding to its complement $1-e_{(H)_G}$. By \cite[Theorem 4.4]{Barnes_Splitting} this gives a monoidal Quillen Equivalence.

\begin{prop}\label{prop:splittingofexep} There is a strong symmetric monoidal Quillen equivalence:
\[
\xymatrix{
\triangle\ :\ L_{S_\bQ}(G -\mathrm{Sp}^\cO)\ \ar@<+1ex>[r] & \ L_{e_{(H)_G} S_\bQ}(G-\mathrm{Sp}^\cO)\times L_{(1-e_{(H)_G}) S_\bQ}(G-\mathrm{Sp}^\cO)\ :\Pi \ar@<+0.5ex>[l]
}
\]
where the left adjoint is a diagonal functor, the right one is a product and the product category on the right is considered with the objectwise model structure (a map $(f_1,f_2)$ is a weak equivalence, a fibration or a cofibration if both factors $f_i$ are).
\end{prop}

From now on we will work only with the category $L_{e_{(H)_G} S_\bQ}(G-\mathrm{Sp}^\cO)$ as this is our model for rational  $G$--spectra over an exceptional subgroup $H$. 

We use the name \begin{em}{$H$--equivalence}\end{em} for a weak equivalences in the category $L_{e_{(H)_G}S_{\bQ}}(G-\mathrm{Sp}^\cO)$ and \begin{em}{$H$--fibrant replacement}\end{em} for the fibrant replacement there. These names are motivated by the following

\begin{lem} A map $f$ between $e_{(H)_G}S_\bQ$ - local objects is a weak equivalence in  $L_{e_{(H)_G}S_{\bQ}}(G-\mathrm{Sp}^\cO)$ if $\pi_\ast^H(f)$ is an isomorphism.
\end{lem}

\section{Change--of--group functors and localizations using idempotents}\label{section:localization}

Since later we will be interested in taking $H$--fixed points of  $G$--spectra when $H$ is not necessary normal in $G$, we need to pass to $N=N_GH$--spectra first. Suppose we have an inclusion $i: N \hookrightarrow G$ of a subgroup $N$ in a group $G$. This gives a pair of adjoint functors at the level of orthogonal spectra (see for example \cite[Section V.2 ]{MandellMay}), namely induction, restriction and coinduction as below (the left adjoint is above the corresponding right adjoint)

\[
\xymatrix@C=6pc{
G-\mathrm{Sp}^\cO\ 
\ar@<-0ex>[r]|-(.5){\ i^*\ }
&
N-\mathrm{Sp}^\cO\  
\ar@/^1pc/[l]^(.5){F_N(G_+,-)}
\ar@/_1pc/[l]_(.5){G_+\wedge_N -}
}
\]

These two pairs of adjoint functors are Quillen pairs and restriction as a right adjoint is used for example when we want to take $H$--fixed points of  $G$--spectra, where $H$ is not a normal subgroup of $G$. The first step then is to restrict to $N_GH$--spectra and then take $H$--fixed points. This is usually done in one go, since the restriction and $H$--fixed points are both right Quillen functors.\bigskip

It is natural to ask when the pair of adjunctions above passes to the localized categories, in our case localized at $e_{(H)_G}S_\bQ$ and $e_{(H)_N}S_\bQ$ respectively.  The answer is related to $H$ being an $N$--good or bad subgroup in $G$. It turns out that the induction -- restriction adjunction does not always induce a Quillen adjunction on the localized categories, unless $H$ is $N$--good in $G$. However, the restriction -- coinduction adjunction induces a Quillen adjunction on the localized categories, for all exceptional subgroups $H$. Before we discuss this particular adjunction we state a general result.

\begin{lem}\label{locAdjAtObject} Suppose that $F: \cC \rightleftarrows \cD: R$ is a Quillen adjunction of model categories where the left adjoint is strong monoidal. Suppose further that $E$ is a cofibrant object in $\cC$ and that both $L_E\cC$ and $l_{F(E)}\cD$ exist. Then 
\[
\xymatrix{
F \ :\ L_{E}\cC \  \ar@<+1ex>[r] & \  L_{F(E)}\cD\ : R \ar@<+0.5ex>[l]
}
\]
is a strong monoidal Quillen adjunction. Moreover if the original adjunction was a Quillen equivalence then the one induced on the level of localized categories is as well.
\end{lem}
\begin{proof}Since the localization didn't change the cofibrations, the left adjoint $F$ still preserves them. To show that it also preserves acyclic cofibrations, take an acyclic cofibration ${f: X\lra Y}$ in $L_E\cC$. By definition $f\wedge Id_E$ is an acyclic cofibration in $\cC$. Since $F$ was a left Quillen functor before localization $F(f\wedge Id_E)$ is an acyclic cofibration in $\cD$. As $F$ was strong monoidal we have $F(f\wedge Id_E)\cong F(f)\wedge Id_{F(E)}$, so $F(f)$ is an acyclic cofibration in $L_{F(E)}\cD$ which finishes the proof of the first part.
\bigskip 

To prove the second part of the statement we use Part 2 from \cite[Corollary 1.3.16]{Hovey}.
Since $F$ is strong monoidal and the original adjunction was a Quillen equivalence $F$ reflects $F(E)$--equivalences between cofibrant objects. It remains to check that the derived counit is an $F(E)$--equivalence. $F(E)$--fibrant objects are fibrant in $\cD$ and the cofibrant replacement functor remains unchanged by localization. Thus this follows from the fact that $F,R$ was a Quillen equivalence.
\end{proof}

We will use this result in several cases for the following two adjoint pairs of $G$ orthogonal spectra. Notice that since both left adjoints are strong monoidal, the results below follow from Lemma \ref{locAdjAtObject}.
\begin{cor}Let $i : N \lra G$ denote the inclusion of a subgroup and let $E$ be a cofibrant object in $G-\mathrm{Sp}^\cO$. Then
\[
\xymatrix{
i^\ast \ :\ L_{E}(G-\mathrm{Sp}^\cO) \  \ar@<+1ex>[r] & \  L_{i^\ast(E)}(N- \mathrm{Sp}^\cO)\ : F_N(G_+,-) \ar@<+0.5ex>[l]
}
\] 
is a strong monoidal Quillen pair.
\end{cor}

\begin{cor}Let $\epsilon : N \lra W$ denote the projection of groups, where $H$ is normal in $N$ and $W=N/H$. Let $E$ be a cofibrant object in $W-\mathrm{Sp}^\cO$. Then
\[
\xymatrix{
\epsilon^\ast \ :\ L_{E}(W-\mathrm{Sp}^\cO) \  \ar@<+1ex>[r] & \  L_{\epsilon^\ast(E)}(N- \mathrm{Sp}^\cO)\ : (-)^H \ar@<+0.5ex>[l]
}
\] 
is a strong monoidal Quillen pair.
\end{cor}

The following two results describe the behaviour of the restriction--induction adjunction at the level of localized categories.

\begin{prop}\label{rightAdjoint1} Suppose $H$ is an exceptional subgroup of $G$ which is $N=N_GH$-good in $G$. Then
\[
\xymatrix{
i^\ast\ :\ L_{e_{(H)_G}S_\bQ}(G-\mathrm{Sp}^\cO)\ \ar@<-1ex>[r] & L_{e_{(H)_N}S_\bQ}(N-\mathrm{Sp}^\cO)\ :\ G_+\wedge_N- \ar@<-0.5ex>[l]
}
\]
is a Quillen adjunction.
\end{prop}
\begin{proof} This was a Quillen adjunction before localization by \cite[Chapter V, Proposition 2.3]{MandellMay} so the left adjoint preserves cofibrations. It preserves acyclic cofibrations as $G_+\wedge_N-$ preserved acyclic cofibrations before localization and we have a natural (in an $N$--spectrum $X$) isomorphism (see \cite[Chapter V, Proposition 2.3]{MandellMay}):
$$(G_+\wedge_N X) \wedge e_{(H)_G}S_\bQ \cong G_+ \wedge_N (X \wedge i^*(e_{(H)_G}S_\bQ))$$
Note that, since $H$ is $N$-good in $G$,  $i^*(e_{(H)_G})\cong e_{(H)_N}$, where the latter is the idempotent corresponding to $(H)_N$ in $\A(N)_\bQ$.

\end{proof}

\begin{prop}\label{notQErest_ind}Suppose $H$ is an exceptional subgroup of $G$ which is $N=N_GH$--bad in $G$. Then
\[
\xymatrix{
i^\ast\ :\ L_{e_{(H)_G}S_\bQ}(G-\mathrm{Sp}^\cO)\ \ar@<-1ex>[r] & L_{e_{(H)_N}S_\bQ}(N-\mathrm{Sp}^\cO)\ :\ G_+\wedge_N- \ar@<-0.5ex>[l]
}
\]
is not a Quillen adjunction.
\end{prop}

\begin{proof}It is enough to show that $G_+\wedge_N-$ does not preserve acyclic cofibrations.
Firstly, since $H$ is $N$--bad in $G$ there exists $H'$ such that $(H)_G=(H')_G$ and $(H)_N\neq (H')_N$.  

Take a map $f$ to be the inclusion into the coproduct $N/H_+ \lra N/H_+ \vee N/H'_+$. This is a weak equivalence in $L_{e_{(H)_N}S_\bQ}(N-\mathrm{Sp}^\cO)$ since $\Phi^H(N/H_+)=\Phi^H(N/H_+ \vee N/H'_+)$. It is also a cofibration as a pushout of a cofibration $\ast \lra N/H_+$ along the map $\ast \lra N/H'_+$. Applying the left adjoint gives the inclusion $G_+\wedge_Nf: G/H_+ \lra G/H_+\vee G/H'_+$. Now $\Phi^H(G/H_+\vee G/H'_+)=N/H_+\vee N/H'_+ \neq N/H_+$ since $H$ is $N$--bad by assumption and $(H)_G=(H')_G$.
\end{proof}

It turns out that the restriction and function spectrum adjunction gives a Quillen adjunction under general conditions. 

\begin{lem}\label{QuillenAdjForIandH} Suppose $G$ is any compact Lie group, $i: N \lra G$ is an inclusion of a subgroup and $V$ is an open and closed $G$-invariant set $V$ in $\mathrm{Sub_f}(G)$ which is a union of $\sim$-equivalence classes (see Section \ref{section:subgroups}). Then the adjunction
\[
\xymatrix{
i^\ast\ :\ L_{e_VS_\bQ}(G-\mathrm{Sp}^\cO)\ \ar@<+1ex>[r] & L_{e_{i^\ast V}S_\bQ}(N-\mathrm{Sp}^\cO)\ :\ F_N(G_+,-) \ar@<+0.5ex>[l]
}
\]
is a Quillen pair. 
\end{lem}
\begin{proof}Before localizations this was a Quillen pair by \cite[Chapter V, Proposition 2.4]{MandellMay}. It is a Quillen pair after localization by Lemma \ref{locAdjAtObject}, and the fact that $i^\ast$ is strong symmetric monoidal. We use the notation $i^\ast V$ for the preimage of $V$ under the inclusion on spaces of subgroups induced by $i$, i.e. $\mathrm{Sub_f}(N) \lra \mathrm{Sub_f}(G) $, see Section \ref{section:subgroups}.
\end{proof}

We will repeatedly use the lemma above, mainly in situations where after further localization of the right hand side we will get a Quillen equivalence. 

\begin{cor}\label{localizedQAdjunctions}
 Suppose $G$ is a compact Lie group and $H$ is an exceptional subgroup of $G$.  Then
\[
\xymatrix{
i^\ast\ :\ L_{e_{(H)_G}S_\bQ}(G-\mathrm{Sp}^\cO)\ \ar@<+1ex>[r] & L_{e_{(H)_N}S_\bQ}(N-\mathrm{Sp}^\cO)\ :\ F_N(G_+,-) \ar@<+0.5ex>[l]
}
\]
is a Quillen adjunction.

\end{cor}
\begin{proof} For $H$ which is $N=N_GH$--good the result follows from the fact that the idempotent on the right hand side $e_{(H)_N}=i^*(e_{(H)_G})=e_{i^*((H)_G)}$. For $H$ which is $N=N_GH$--bad it is true since the left hand side is a further localization of $L_{e_{i^*((H)_G)}S_\bQ}(N-\mathrm{Sp}^\cO)$ at the idempotent $e_{(H)_N}$:
\[
\xymatrix@C=3pc{
L_{e_{(H)_G}S_\bQ}(G-\mathrm{Sp}^\cO)\
\ar@<+1ex>[r]^{i^\ast}
&
\ L_{i^*(e_{(H)_G}) S_{\bQ}}(N- \mathrm{Sp}^\cO)\
\ar@<+0.5ex>[l]^{F_{N}(G_+,-)}
\ar@<+1ex>[r]^(.53){\mathrm{Id}}
&
\ L_{e_{(H)_N}S_\bQ}(N-\mathrm{Sp}^\cO)
\ar@<+0.5ex>[l]^(.47){\mathrm{Id}}
}
\]
Note that since $H$ is $N$--bad, $e_{(H)_N} \neq i^*(e_{(H)_G})$ and $e_{(H)_N} i^*(e_{(H)_G})=e_{(H)_N}$.

\end{proof}

In the next two theorems we show that the Quillen adjunction above is in fact a Quillen equivalence.

\begin{thm}\label{theorem111} Suppose $N=N_GH$ and $H$ is an exceptional subgroup of $G$ that is $N$--good. Then the adjunction
\[
\xymatrix{
i^\ast \ :\ L_{e_{(H)_G}S_{\bQ}}(G-\mathrm{Sp}^\cO) \  \ar@<+1ex>[r] & \  L_{e_{(H)_N} S_{\bQ}}(N- \mathrm{Sp}^\cO)\ : F_N(G_+,-) \ar@<+0.5ex>[l]
}
\]
is a strong symmetric monoidal Quillen equivalence. 
\end{thm}

\begin{proof} Firstly, if $H$ is an $N$--good exceptional subgroup of $G$ with an idempotent $e_{(H)_G}$ then $e_{(H)_N}=i^\ast(e_{(H)_G})$ in $\A(N)_\bQ$. 

This is a Quillen adjunction by Corollary \ref{localizedQAdjunctions} and we claim that $i^\ast$ preserves all $H$--equivalences. Suppose $f:X \lra Y$ is an $H$--equivalence in $L_{e_{(H)_G}S_{\bQ}}(G-\mathrm{Sp}^\cO)$, i.e. $Id_{e_{(H)_G}S_\bQ}\wedge f: e_{(H)_G}S_\bQ \wedge X \lra e_{(H)_G}S_\bQ \wedge Y$ is a $\pi_\ast$--isomorphism. As $i^\ast$ is strong monoidal $$i^\ast(Id_{e_{(H)_G}S_\bQ}\wedge f) \cong Id_{i^\ast(e_{(H)_G}S_\bQ)}\wedge i^\ast(f) \cong Id_{e_{(H)_N}S_\bQ}\wedge i^\ast(f)$$ and $i^\ast$ preserves $\pi_\ast$--isomorphisms we can conclude.

To show this is a Quillen equivalence we will use Part 2 from \cite[Corollary 1.3.16]{Hovey}.
It is easy to see that $i^\ast$ reflects $H$--equivalences using the fact it is strong monoidal and the isomorphism $[N/H_+, i^\ast(X)]^N\cong [G/H_+,X]^G$.

As $i^\ast$ preserves all $H$--equivalences it is enough to check that for every fibrant $Y \in L_{e_{(H)_N} S_{\bQ}}(N- \mathrm{Sp}^\cO)$ the counit map $\varepsilon_Y:i^\ast F_N(G_+,Y) \lra Y$ is an $H$--equivalence (in $N$--spectra), i.e it is a $\pi_\ast^H$--isomorphism of $N$--spectra.

First we check that domain and codomain have isomorphic stable $H$ homotopy groups:
\begin{multline}
\pi_\ast^H(i^*F_N(G_+,Y))\cong \pi_\ast^H(F_N(G_+,Y)) \cong [G/H+, F_N(G_+,Y)]_\ast^G \\
\cong [i^\ast(G/H_+),Y]_\ast^N\cong [N/H_+,Y]_\ast^N\cong \pi_\ast^H(Y)
\end{multline}
The next-to-last isomorphism follows from the fact that the map $N/H_+ \lra G/H_+$ (induced by inclusion $N \lra G$) is an $H$--equivalence in $N$--spectra, i.e an equivalence in $L_{e_{(H)_N} S_{\bQ}}(N- \mathrm{Sp}^\cO)$.

By Lemma \ref{qequiv_generators} it is enough to check the counit condition for a generator.   We will check it for the spectrum $i^\ast (\hat{f}G/H_+)$, which is a compact generator for localized $N$--spectra (it is $H$--equivalent to $N/H_+$). The stable $H$--homotopy groups of this generator are $\bQ[W_{G}H]$ in degree $0$ (where $W_{G}H$ is the Weyl group for $H$ in $G$, so in particular $\bQ[W_{G}H]$ is a finite dimensional vector space by assumption that $H$ is exceptional in $G$) and $0$ in other degrees.

Now it is enough to show that $[N/H_+,\varepsilon_{i^\ast (\hat{f}G/H_+)}]^N$ is surjective. One of the triangle identities on $i^\ast (\hat{f}G/H_+)$ for the adjunction requires that the following diagram commutes 

\[
\xymatrix@R=2pc@C=5pc{
i^*(\hat{f}G /H_+) \ar[dr]^{\mathrm{Id}} \ar[d]_{i^\ast(\eta_{\hat{f}G/H_+})} &   \\
  {i^* F_N(G_+, i^*(\hat{f}G /H_+))} \ar[r]_(.57){\varepsilon_{i^\ast (\hat{f}G/H_+)}} & {i^*(\hat{f}G /H_+)} \\ 
}
\]

Thus postcomposition with $\varepsilon_{i^\ast (\hat{f}G/H_+)}$ is surjective on the homotopy level. It follows that the counit map is an $H$--equivalence of $N$--spectra for every fibrant $Y$, which finishes the proof.
\end{proof}

The argument above will not work in the context where $i^\ast$ does not preserve fibrant replacements. However we found the proof above amusing, so we decided to present it, even though the proof below can be applied also in the case where $H$ is an exceptional $N_GH$--good subgroup of $G$.

\begin{thm}\label{badLeftAdjoint} Suppose $H$ is an exceptional subgroup of $G$. Then the composite of adjunctions
\[
\xymatrix@C=3pc{
L_{e_{(H)_G}S_{\bQ}}(G-\mathrm{Sp}^\cO)\
\ar@<+1ex>[r]^{i^\ast_N}
&
\ L_{i^\ast(e_{(H)_G} S_{\bQ})}(N- \mathrm{Sp}^\cO)\
\ar@<+0.5ex>[l]^{F_N(G_+,-)}
\ar@<+1ex>[r]^(.54){\mathrm{Id}}
&
\ L_{e_{(H)_N} S_{\bQ}}(N- \mathrm{Sp}^\cO)
\ar@<+0.5ex>[l]^(.46){\mathrm{Id}}
}
\]
is a strong symmetric monoidal Quillen equivalence, where $e_{(H)_N}$ denotes the idempotent of the rational Burnside ring $\A(N)_\bQ$ corresponding to the characteristic function of $(H)_N$. Notice that if $H$ is $N$--good then the right adjunction is trivial.
\end{thm}
\begin{proof} Firstly, if $H$ is $N$--bad then $i^\ast_N(e_{(H)_G}S_\bQ)\not\simeq e_{(H)_N}S_\bQ$ as localized $N$--spectra. The reason for that is that $(H)_G$ restricts to more than one conjugacy class of subgroups of $N$. That is why we need a further localization - we only want to consider $(H)_N$.

The composite above forms a Quillen adjunction by Corollary \ref{localizedQAdjunctions}. We use Part 3 from \cite[Corollary 1.3.16]{Hovey}  to show that it is a Quillen equivalence.
Observe that $F_N(G_+,-)$ preserves and reflects weak equivalences between fibrant objects. Let $X$ be a fibrant object in $L_{e_{(H)_N} S_{\bQ}}(N- \mathrm{Sp}^\cO)$. Then $F_N(G_+, X)$ is also fibrant and
\begin{multline}
[G/H_+,e_{(H)_G}F_N(G_+,X)]^{G} \cong [G/H_+, F_N(G_+, X)]^G \\
\cong [i^\ast_N(G/H_+),X]^N\cong [e_{(H)_N}i^\ast_N(G/H_+),e_{(H)_N}X]^N\cong [N/H_+,e_{(H)_N}X]^N
\end{multline} 
Now we need to show that the derived unit is a weak equivalence on the cofibrant generator for $L_{e_{(H)_G}S_{\bQ}}(G-\mathrm{Sp}^\cO)$, which is $e_{(H)_G}G/H_+$. This is 
$$e_{(H)_G}G/H_+ \lra F_N(G_+,e_{(H)_N}i^\ast_N(e_{(H)_G}G/H_+))$$
To check that this is a weak equivalence in $L_{e_{(H)_G}S_{\bQ}}(G-\mathrm{Sp}^\cO)$ it is enough to check that on the homotopy level the induced map
$$[G/H_+, e_{(H)_G} G/H_+]^G \lra [G/H_+,F_N(G_+,e_{(H)_N}i^\ast_N (e_{(H)_G} G/H_+))]^G $$
is an isomorphism. This map fits into a commuting diagram below

\[
\xymatrix@R=2pc@C=1.8pc{
 [G /H_+, e_{(H )_G}G/H_+]^G \ar[dr]^{Li^\ast_N} \ar[d] &  \\
  {[G/H_+, F_N(G_+,e_{(H)_N}i^\ast_N (e_{(H)_G} G/H_+))]^G} \ar[r]^(.53){\cong}& {[i^\ast_N G/H_+, e_{(H)_N}i^\ast_N (e_{(H)_G}G/H_+)]^N} \\ 
}
\]

Since the horizontal map is an isomorphism it is enough to show that $Li^\ast_N$ is an isomorphism.
This follows from the commutative diagram:

\[
\xymatrix@R=1.7pc@C=2pc{
& [S^0, i^*_H (e_{(H)_N}i^*_N(e_{(H)_G}G/H_+))]^H \\
[G/H_+, e_{(H)_G}G/H_+]^G \ar[r]^{\cong} \ar[d] \ar@<-10ex>@/_2.5pc/[dd]_{Li^*_N} & [S^0, i^*_H (e_{(H)_G}G/H_+)]^H \ar[u]^{\cong}  \\
 [i^*_NG/H_+, i^*_N(e_{(H)_G}G/H_+)]^N \ar[r]^{j^*} \ar[d] & [N/H_+, i^*_N (e_{(H)_G}G/H_+)]^N \ar[u]^{\cong} \ar[d] \\
 [i^*_N G/H_+, e_{(H)_N}i^*_N(e_{(H)_G}G/H_+)]^N \ar[r]_{j^*}^{\cong} & [N/H_+, e_{(H)_N}i^*_N(e_{(H)_G}G/H_+)]^N \ar@<-10ex>@/_2.5pc/[uuu]_{\cong} \\
}
\]
where 
$j: N/H_+ \lra i^*_NG/H_+$ is a weak equivalence in $L_{e_{(H)_N}S_\bQ}N-\mathrm{Sp}^\cO$
and $i^*_H$ denotes the restriction functor from $G$--spectra to $H$--spectra.

\end{proof}


\section{A monoidal algebraic model for rational $G$--spectra over an exceptional subgroup}\label{section:monoidal}
The category of rational  $G$--spectra over an exceptional subgroup $H$ is modelled by the left Bousfield localization at an idempotent $e_{(H)_G}$ corresponding to the conjugacy class of $H$ in $G$. Thus from now on we will use the notation $H$ for an exceptional subgroup of $G$, and we will work with the category $L_{e_{(H)_G} S_{\bQ}}(G - \mathrm{Sp}^\cO)$. 
 
The main difference between this approach and what appears in the literature is in replacing Morita equivalence by the inflation--fixed point adjunction. This became possible after analysing an interplay of induction--restriction--coinduction adjunctions with left Bousfield localizations in Section \ref{section:localization}.

The plan for the zig-zag of symmetric monoidal Quillen equivalences is as follows. First we move from the category ${L_{e_{(H)_G}S_{\bQ}}(G- \mathrm{Sp}^\cO)}$ to the category $L_{e_{(H)_N} S_{\bQ}}(N- \mathrm{Sp}^\cO)$ using the restriction--coinduction adjunction. Recall that $N$ denotes the normalizer $N_GH$. 

The second step is to use the inflation--fixed point adjunction between $L_{e_{(H)_N} S_{\bQ}}(N- \mathrm{Sp}^\cO)$  and $L_{e_1 S_{\bQ}}(W- \mathrm{Sp}^\cO)$, where $W$ denotes the Weyl group $N/H$. Recall that $W$ is finite, as $H$ is an exceptional subgroup of $G$ and $e_1$ denotes the idempotent in $\A(W)_\bQ$ corresponding to the trivial subgroup.

Next we use the restriction of universe to pass from $L_{e_1 S_{\bQ}}(W- \mathrm{Sp}^\cO)$ to the category $\mathrm{Sp}^\cO_\bQ[W]$ of rational orthogonal spectra with $W$--action.  Note that we could have combined the two steps above into one, since both left adjoints point the same way, however for the clarity of the arguments we decided to treat them separately.

Now we pass to symmetric spectra with $W$--action using the forgetful functor from orthogonal spectra. Next we move to $H\bQ$-modules with $W$--action in symmetric spectra. From here we use the result of \cite[Theorem 1.1]{ShipleyHZ} to get to $\Ch(\bQ)[W]$, the category of rational chain complexes with $W$--action, which is equivalent as monoidal model category to $\Ch(\bQ[W])$, the category of chain complexes of $\bQ[W]$-modules with a projective model structure. That gives an algebraic model which is compatible with the monoidal product, i.e. the zig-zag of our Quillen equivalences induces a strong symmetric monoidal equivalence on the level of homotopy categories. 

To illustrate the whole path we present a diagram which shows every step of this comparison. The reader may wish to refer to this diagram now, but the notation will be introduced as we proceed. Left Quillen functors are placed on the left. Recall that $N=N_GH$ and $W=W_GH=N_GH/H$.

\[
\xymatrix@R=1.5pc{
L_{e_{(H)_G}S_{\bQ}}(G-\mathrm{Sp}^\cO)
\ar@<-1ex>[d]_{i^{\ast}}
\\
L_{e_{(H)_N}S_{\bQ}}(N-\mathrm{Sp}^\cO)
\ar@<-1ex>[u]_{F_N(G_+,-)}
\ar@<+1ex>[d]^{(-)^H}
\\
L_{e_1S_\bQ}(W-\mathrm{Sp}^\cO)
\ar@<+1ex>[u]^{ \epsilon^\ast}
\ar@<+1ex>[d]^{\mathrm{res}}
\\
\mathrm{Sp}^\cO_\bQ[W]
\ar@<+1ex>[u]^{L}
\ar@<+1ex>[d]^{ \mathrm{Sing}\circ \mathbb{U}}
\\
Sp_\bQ^\Sigma[W]
\ar@<+1ex>[u]^{\mathbb{P} \circ |-|}
\ar@<-1ex>[d]_{H\bQ\wedge -}
\\
(H\bQ-\mathrm{mod})[W]
\ar@<-1ex>[u]_{U}
\ar@<-0ex>[d]_{\mathrm{zig-zag\ of}}
\\
\Ch(\bQ[W])
\ar@<-0ex>[u]_{\mathrm{Quillen\ equivalences}}
}
\]

\subsection{The category $\Ch(\bQ[W]-\mathrm{mod})$}\label{chainExcep}
Before we start describing the zig zag of Quillen equivalences towards the algebraic model for rational  $G$--spectra over an exceptional subgroup, we briefly describe the algebraic model. Suppose $W$ is a finite group. In this section we discuss the category of chain complexes of left $\bQ[W]$ modules.  

Firstly, this category may be equipped with the projective model structure, where weak equivalences are homology isomorphisms and fibrations are levelwise surjections. Cofibrations are levelwise split monomorphisms with cofibrant cokernel. This model structure is cofibrantly generated by \cite[Section 2.3]{Hovey}. 

Note that $\bQ[W]$ is not generally a commutative ring, however it is a Hopf algebra with cocommutative coproduct given by $$\Delta : \bQ[W] \lra \bQ[W] \otimes \bQ[W] \ \ ,\ \ g \mapsto g \otimes g.$$ This allows us to define an associative and commutative tensor product on $\Ch(\bQ[W]-\mathrm{mod})$, namely tensor over $\bQ$, where the action on the $X \otimes_\bQ Y$ is diagonal. The unit is a chain complex with $\bQ$ at the level 0 with trivial $W$--action and zeros everywhere else and it is cofibrant in the projective model structure. The monoidal product defined this way is closed, where the internal hom is given by an internal hom over $\bQ$ with $W$--action given by conjugation.  

This category is equivalent (as a monoidal model category) to the category of $W$--objects in a category of  $\Ch(\bQ-\mathrm{mod})$, with the projective model structure, i.e. this is a model structure which is a transfer of the projective model structure on $\Ch(\bQ-\mathrm{mod})$ to the category of $W$ objects there, using the forgetful functor as a right adjoint. 

It is shown in \cite[Proposition 4.3]{BarnesFiniteG} that this is a monoidal model category satisfying the monoid axiom.
Now we are ready to establish the zig-zag of Quillen equivalences.

\subsection{Monoidal comparison}\label{monoidalH}

At the beginning of this approach we would like to use the inflation--fixed point adjunction. However, as $H$ is not necessary normal in $G$ first we need to move to the category of $N$--orthogonal spectra, where $N=N_GH$. Notice that for our purpose this passage needs to be monoidal.

The inclusion of a subgroup $i: N \lra G$ induces two adjoint pairs between corresponding categories of orthogonal spectra that we discussed earlier. The first choice would be to work with the induction and restriction adjunction. However, in case of our localizations, this is not always a Quillen adjunction as we discussed in detail in Section \ref{section:localization}. The restriction functor $i^{\ast}$ is strong monoidal, so we choose to work with it as a left adjoint, where the right adjoint is the coinduction functor. We showed in Section \ref{section:localization} that this is always a strong monoidal Quillen adjunction for localizations at idempotents corresponding to conjugacy classes of exceptional subgroups.  By Theorem \ref{badLeftAdjoint} it is a Quillen equivalence, so we get the first step of the zig-zag:

\begin{thm} Suppose $H$ is an exceptional subgroup of $G$. Then the composite of adjunctions
\[
\xymatrix@C=3pc{
L_{e_{(H)_G}S_{\bQ}}(G-\mathrm{Sp}^\cO)\
\ar@<+1ex>[r]^{i^\ast}
&
\ L_{i^\ast(e_{(H)_G} S_{\bQ})}(N- \mathrm{Sp}^\cO)\
\ar@<+0.5ex>[l]^{F_N(G_+,-)}
\ar@<+1ex>[r]^(.53){\mathrm{Id}}
&
\ L_{e_{(H)_N} S_{\bQ}}(N- \mathrm{Sp}^\cO)
\ar@<+0.5ex>[l]^(.47){\mathrm{Id}}
}
\]
is a strong symmetric monoidal Quillen equivalence. Notice that if $H$ is $N$--good then $i^*(e_{(H)_G})=e_{(H)_N}$ and the right adjunction is trivial.
\end{thm}

Now we use the inflation--fixed point adjunction. Recall that $W$ below denotes the Weyl group $N_GH/H$ and by the assumption on $H$ it is finite. Moreover there is a projection map ${\epsilon: N\lra W}$ which induces the left adjoint below.

\begin{thm}The adjunction
\[
\xymatrix{
\epsilon^{\ast} \ :\ L_{e_1 S_{\bQ}}(W- \mathrm{Sp}^\cO) \  \ar@<+1ex>[r] & \  L_{e_{(H)_N} S_{\bQ}}(N- \mathrm{Sp}^\cO)\ : (-)^H \ar@<+0.5ex>[l]
}
\]
is a strong monoidal Quillen equivalence. Here $e_1$ is the idempotent of the rational Burnside ring $\A(W)_\bQ$ corresponding to the characteristic function for the trivial subgroup. 
\end{thm}
\begin{proof}  To prove this is a Quillen pair we refer to \cite[Proposition 3.2]{GreenShipleyFixedPoints} which states that (in notation adapted to our case):
\[
\xymatrix{
\epsilon^{\ast} \ :\ (W- \mathrm{Sp}^\cO) \  \ar@<+1ex>[r] & \  L_{\widetilde{E}[ \not\supseteq H]}(N- \mathrm{Sp}^\cO)\ : (-)^H \ar@<+0.5ex>[l]
}
\]
is a Quillen equivalence. Recall that $\widetilde{E}[ \not\supseteq H]$ is a cofibre of a map $E[ \not\supseteq H] \lra S^0$ where $[ \not\supseteq H]$ denotes the family of subgroups of $N$ not containing $H$. Now we localize this result further at $e_1S_\bQ$ on the side of $W$--spectra and $e_{(H)_N}S_\bQ$ on the side of $N$--spectra. It follows from Lemma \ref{locAdjAtObject} and the fact that $\epsilon^*$ is strong monoidal that the resulting adjunction is a Quillen equivalence. The right hand side after this localization is just  ${L_{e_{(H)_N} S_{\bQ}}(N- \mathrm{Sp}^\cO)}$.
\end{proof}

Next we move from  $L_{e_1 S_{\bQ}}(W- \mathrm{Sp}^\cO)$ to $\mathrm{Sp}^\cO_\bQ[W]$ (where $\mathrm{Sp}^\cO_\bQ=L_{S_\bQ}\mathrm{Sp}^\cO$) using the restriction and extension of $W$--universe from the complete to the trivial one. 
From now on we will work with $W$--objects in a category $\cC$, where $W$ is a finite group. We denote this category by $\cC[W]$. We can think of $\cC[W]$ as a category of functors from $W$, which is a one object category with $Hom(\ast, \ast)=W$ to $\cC$, also known as $\cC^W$. The inclusion $j$ of a terminal category $\mathsf{1}$ into $W$ gives two adjoint pairs $(\lan_j, j^*)$ and $(j^*, \ran_j)$. We will use notation $U$ for $j^*$. It turns out that if $\cC$ is a cofibrantly generated model category, then $\cC[W]$ can be equipped with a model structure by applying transfer \cite[Theorem 11.3.2]{Hirschhorn} to the adjunction below:
\[
\xymatrix{
\lan_j\ :\ \cC\  \ar@<+1ex>[r] & \  \cC[W]\ : \ U \ar@<+0.5ex>[l]
}
\]
Here $\lan_j$ is the left Kan extension along $j$. Notice, that in this case it is sending $X$ to a coproduct of $X$ indexed by elements of $W$, with $W$ acting by permuting the factors. It is a straightforward observation that $U$ preserves cofibrations, as generating cofibrations in $\cC[W]$ are just images of the generating cofibrations in $\cC$ under $\lan_j$. 

If $\cC$ is a closed symmetric monoidal model category then $\cC[W]$ is as well, by analogous observations to those in Section \ref{chainExcep}. Notice that the monoidal product on $W$--objects in $\cC$ is the one from $\cC$ with the diagonal $W$--action and $U_\cC$ is strong monoidal. It is enough to check the pushout--product axiom in $\cC[W]$ for generating cofibrations and acyclic cofibrations, and since they are the images of the generating cofibrations and acyclic cofibrations (respectively) under $\lan_j$ the pushout--product axiom follows from the one in $\cC$. The unit axiom follows from the unit axiom in $\cC$ and the fact that $U$ preserves cofibrations.

\begin{lem}\label{restrictionTotrivialuniverse}The adjunction
\[
\xymatrix{
\mathrm{I}_t^c \ :\ \mathrm{Sp}^\cO_\bQ[W] \  \ar@<+1ex>[r] & \  L_{e_1 S_{\bQ}}(W- \mathrm{Sp}^\cO)\ : \mathrm{I}_c^t=\mathrm{res} \ar@<+0.5ex>[l]
}
\]
is a strong monoidal Quillen equivalence. We use $\mathrm{I}_c^t$ to denote the restriction (denoted also $\mathrm{res}$ above) from the complete $W$--universe to the trivial one. $\mathrm{I}_t^c$ denotes the extension from the trivial $W$--universe to the complete one.
\end{lem}
\begin{proof} This is a strong monoidal adjunction by \cite[Chapter V, Theorem 1.5]{MandellMay}. Now we note that the left adjoint preserves generating cofibrations and generating acyclic cofibrations, since $\mathrm{I}_t^c F_V \cong F_V$ by \cite[Chapter V, 1.4]{MandellMay}.
 
The right adjoint $\mathrm{res}$ preserves and reflects all weak equivalences  since in both model structures they are defined as those maps which after forgetting to non equivariant spectra are rational $\pi_*$--isomorphisms. The derived unit for the cofibrant generator $W_+$ (in this case categorical unit is also the derived unit) is an isomorphism which follows from \cite[Chapter V, Theorem 1.5]{MandellMay}, and thus for any cofibrant object it is a weak equivalence. By Part 3 of \cite[Corollary 1.3.16]{Hovey} this is a Quillen equivalence.
\end{proof} 

We removed all difficulties coming from the equivariance with respect to a topological group. What is left now is a finite group action on the rational orthogonal spectra. 

For the remaining Quillen equivalences we will need the following, well--known fact.

\begin{prop}\label{Wobjects}Suppose 
\[
\xymatrix{
F\ :\ \cC \  \ar@<+1ex>[r]  & \  \cD\ : \ G \ar@<+0.5ex>[l] 
}
\] 
is a Quillen equivalence and $W$ is a finite group. Then this adjunction induces a Quillen equivalence at the level of $W$--objects in $\cC$ and $\cD$ (with model structures transferred from that on $\cC$ and $\cD$ respectively). Moreover if $(F,G)$ is a weak monoidal Quillen equivalence between monoidal model categories then it is so when induced to the level of $W$--objects in $\cC$ and $\cD$.
\end{prop}
\begin{proof}We have the following diagram
\[
\xymatrix@C=3pc@R=3pc{\cC[W] \ar[d]^{U_\cC} \ar@<+1ex>[r]^{F_W} &\cD[W] \ar[d]^{U_\cD} \ar@<+0.5ex>[l]^{G_W}
\\ \cC \ar@<+1ex>[r]^{F} & \cD \ar@<+0.5ex>[l]^{G}} 
\]
where the functors $U_\cC$ and $U_\cD$ commute with both left and right adjoints. Moreover $U_\cC$ and $U_\cD$ create weak equivalences and fibrations and they preserve cofibrant objects (they preserve cofibrations and initial objects) and fibrant objects. A check of the condition from the definition of Quillen equivalence for the adjunction $(F_W,G_W)$ is just a diagram chase.

For the monoidal consideration, recall that the monoidal product on $W$--objects in $\cC$ is the one from $\cC$ with the diagonal $W$--action and $U_\cC$ is strong monoidal. If $(F,G)$ is a weak monoidal Quillen pair then it is again a diagram chase to show that $(F_W,G_W)$ is also a weak monoidal Quillen pair.
\end{proof}

To  apply the result of \cite[Theorem 1.1]{ShipleyHZ} and pass to the category of chain complexes we need to work with rational symmetric spectra in the form of $\mathrm{H}\bQ$-modules (where $H\bQ$ is the Eilenberg--MacLane spectrum for $\bQ$). We pass to this category using the next two lemmas. Both follow from Proposition \ref{Wobjects} and corresponding known results for spectra (See for example Section 7 in \cite{SchwShipEquiv} and recall that $H\bQ$ is weakly equivalent to $S_\bQ$). First we pass to symmetric spectra with a $W$--action using the composition of forgetful functor and the functor induced by singular complex:
\begin{lem}The adjunction
\[
\xymatrix{
\mathbb{P} \circ |-|\ :\ Sp_\bQ^\Sigma[W]\  \ar@<+1ex>[r] & \ \mathrm{Sp}^\cO_\bQ[W] \ : \mathrm{Sing}\circ \mathbb{U} \ar@<+0.5ex>[l]
}
\]
is a strong symmetric monoidal Quillen equivalence. 
\end{lem}

Next we move to $H\bQ$-modules in symmetric spectra with $W$--action.
\begin{lem}\label{SymSpecRationalEquiOrthSpecRational}The adjunction
\[
\xymatrix{
H\bQ \wedge -\ :\ Sp_\bQ^\Sigma[W]\  \ar@<+1ex>[r] & \ (H\bQ-\mathrm{mod})[W] \ : U \ar@<+0.5ex>[l]
}
\]
is a strong symmetric monoidal Quillen equivalence.
Here $U$ denotes forgetful functor and the model structure on $H\bQ-\mathrm{mod}$ is the one created from $Sp^\Sigma$ by the right adjoint $U$.
\end{lem}

From here we use the result of  \cite[Theorem 1.1]{ShipleyHZ} for $R=\bQ$ and Proposition \ref{Wobjects} to get to $\Ch(\bQ)[W]$ with the projective model structure, which is equivalent as a monoidal model category to $\Ch(\bQ[W])$ with the projective model structure (see Section \ref{chainExcep}).

\begin{lem}There is a zig-zag of monoidal Quillen equivalences between the category $(H\bQ-\mathrm{mod})[W]$ and the category $\Ch(\bQ[W])$ with the projective model structure.
\end{lem}

We can summarise the results of this section in the Theorem below.
\begin{thm}\label{monoidalExceptional} There is a zig-zag of symmetric monoidal Quillen equivalences from ${L_{e_{(H)_G}S_{\bQ}}(G- \mathrm{Sp}^\cO)}$ to $\Ch(\bQ[W]-\mathrm{mod})$ with the projective model structure, where $W=N_GH/H$.
\end{thm}

An example of the application of the result above is to the rational $SO(3)$--spectra over an exceptional subgroup in \cite{KedziorekSO(3)}.

\subsection{Finite $G$}\label{section:finiteG}

If $G$ is finite then every subgroup of $G$ is exceptional and there are finitely many conjugacy classes of subgroups of $G$, thus by splitting result of \cite[Theorem 4.4]{Barnes_Splitting} and Proposition \ref{prop:splittingofexep} the category of rational $G$--spectra splits as a finite product of categories, each localized at an idempotent corresponding to the conjugacy class of a subgroup of $G$. 

\begin{prop}\label{prop:splitting_finite}Suppose $G$ is a finite group. Then there is a strong symmetric monoidal Quillen equivalence:
\[
\xymatrix{
\triangle\ :\ L_{S_\bQ}(G -\mathrm{Sp}^\cO)\ \ar@<+1ex>[r] & \ \prod_{(H)_G, H \leq G} L_{e_{(H)_G} S_\bQ}(G-\mathrm{Sp}^\cO)\ :\Pi \ar@<+0.5ex>[l]
}
\]
where the left adjoint is a diagonal functor, the right one is a product and the product category on the right is considered with objectwise model structure.
\end{prop}

This observation allows us to deduce the following

\begin{cor}\label{finite_G_algmod}Suppose $G$ is a finite group. Then there is a zig-zag of symmetric monoidal Quillen equivalences from $L_{S_\bQ}(G- \mathrm{Sp}^\cO)$ to $$\prod_{(H)_G,H\leq G}\Ch(\bQ[W_GH]-\mathrm{mod}).$$
\end{cor}
\begin{proof}This follows from Proposition \ref{prop:splitting_finite} and Theorem \ref{monoidalExceptional}.
 \end{proof}

We remark that this is not a new result, as for a finite group $G$ an algebraic model was given in \cite{SchwedeShipleyMorita} and monoidal consideration was presented in \cite{BarnesFiniteG}. However, the use of localizations of commutative ring $G$--spectra in the proof of \cite{BarnesFiniteG} requires adaptations from \cite{BlumbergHillL1}. Our proof avoids these issues.

\end{document}